\documentclass[12pt]{amsart}

\pdfoutput=1

\usepackage[text={420pt,660pt},centering]{geometry}

\usepackage{color}
\usepackage{stmaryrd}
\usepackage[mathscr]{euscript}
\usepackage{esint,amssymb}
\usepackage{graphicx}
\usepackage{MnSymbol}
\usepackage{mathtools}
\usepackage[colorlinks=true, pdfstartview=FitV, linkcolor=blue, citecolor=blue, urlcolor=blue,pagebackref=false]{hyperref}
\usepackage{microtype}

\usepackage{bm}
\usepackage{scalerel} %for scaling the boxes 
\usepackage{dsfont}
\usepackage[font={footnotesize}]{caption}
\usepackage{tikz}
\usetikzlibrary{arrows}

\definecolor{darkgreen}{rgb}{0,0.5,0}
\definecolor{darkblue}{rgb}{0,0,0.7}
\definecolor{darkred}{rgb}{0.9,0.1,0.1}

\newtheorem{proposition}{Proposition}
\newtheorem{theorem}[proposition]{Theorem}
\newtheorem{lemma}[proposition]{Lemma}
\newtheorem{corollary}[proposition]{Corollary}

\theoremstyle{definition}
\newtheorem{remark}[proposition]{Remark}

\newcommand{\cref}[1]{Corollary~\ref{c.#1}}

\numberwithin{equation}{section}
\numberwithin{proposition}{section}
\newcommand{\Z}{\mathbb{Z}}
\newcommand{\N}{\mathbb{N}}
\newcommand{\R}{\mathbb{R}}

\newcommand{\E}{\mathbb{E}}
\newcommand{\C}{\mathcal{C}}
\renewcommand{\P}{\mathbb{P}}
\newcommand{\F}{\mathcal{F}}

\newcommand{\eps}{\varepsilon}

\renewcommand{\subset}{\subseteq}

\renewcommand{\subset}{\subseteq}

\DeclareMathOperator{\cov}{cov}

\renewcommand{\tilde}{\widetilde}

\renewcommand{\L}{\mathcal{L}}

\renewcommand{\H}{\mathcal{H}}

\usetikzlibrary{decorations.pathreplacing,decorations.markings}

\tikzset{
  on each segment/.style={
    decorate,
    decoration={
      show path construction,
      moveto code={},
      lineto code={
        \path [#1]
        (\tikzinputsegmentfirst) -- (\tikzinputsegmentlast);
      },
      curveto code={
        \path [#1] (\tikzinputsegmentfirst)
        .. controls
        (\tikzinputsegmentsupporta) and (\tikzinputsegmentsupportb)
        ..
        (\tikzinputsegmentlast);
      },
      closepath code={
        \path [#1]
        (\tikzinputsegmentfirst) -- (\tikzinputsegmentlast);
      },
    },
  },
  mid arrow/.style={postaction={decorate,decoration={
        markings,
        mark=at position .5 with {\arrow[#1]{stealth}}
      }}},
}

\begin{document}
\title{A central limit theorem for square ice}
\author[W. Wu]{Wei Wu}
\address[W. Wu]{Mathematics Department, NYU Shanghai \& NYU-ECNU Institute of Mathematical Sciences,  China}
\email{wei.wu@nyu.edu}
\maketitle

\begin{abstract}
We prove that the height function associated with the uniform six-vertex model (or equivalently, the uniform homomorphism height function from $\Z^2$ to $\Z$) satisfies a central limit theorem,  upon some logarithmic rescaling. 
\end{abstract}

\section{introduction}

The six vertex model was initially proposed by Pauling \cite{Pau} as an ideal model for ice.  The square ice, which is the six vertex model on the two dimensional square lattice,  can be defined using arrow configurations on the edges that satisfies the ice rule:  at each vertex there is an equal number of incoming and outcoming arrows.  This leads to six possible configurations at each vertex.  The six-vertex model with weights $a, b,c$ can be defined by giving the probability weight proportional to $a^{n_1+n_2} b^{n_3+n_4} c^{n_5+n_6}$ ,  where $n_i$ is the number of vertices of configuration type $i$, for every configuration that satisfies the ice rule.  

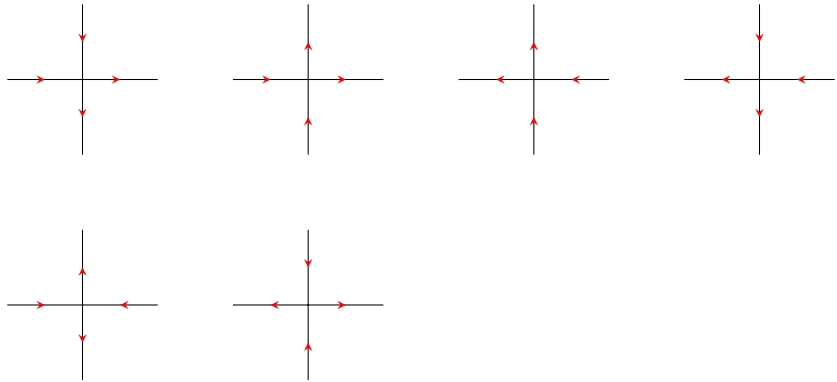
\begin{figure}[h]
\begin{center}
\begin{tikzpicture}
 \path [draw= black, postaction={on each segment={mid arrow=red}}]
 (2,0) -- (3,0)
(3,0) -- (4,0)
(3,-1) --  (3,0)
(3,0) --  (3,1);

\path [draw= black, postaction={on each segment={mid arrow=red}}]
 (6,0) -- (5,0)
(7,0) -- (6,0)
(6,-1) --  (6,0)
(6,0) --  (6,1);

\path [draw= black, postaction={on each segment={mid arrow=red}}]
 (9,0) -- (8,0)
(10,0) -- (9,0)
(9,0) --  (9,-1)
(9,1) --  (9,0);

\path [draw= black, postaction={on each segment={mid arrow=red}}]
 (-1,0) -- (0,0)
(0,0) -- (1,0)
(0,0) --  (0,-1)
(0,1) --  (0,0);

\path [draw= black, postaction={on each segment={mid arrow=red}}]
 (-1,-3) -- (0,-3)
(1,-3) -- (0,-3)
(0,-3) --  (0,-4)
(0,-3) --  (0,-2);

\path [draw= black, postaction={on each segment={mid arrow=red}}]
 (3,-3) -- (2,-3)
(3,-3) -- (4,-3)
(3,-4) --  (3,-3)
(3,-2) --  (3,-3);
\end{tikzpicture}
\end{center}
\caption{The six possible configurations,  each has a weight $a,a,b,b$ and $c, c$.}
%\label{fig:YNd}
\end{figure}

A six-vertex configuration naturally induces a height function,  defined on the dual square lattice,  by the following rule.  Let $e^*:=(u,v)$ be two adjacent vertices on the dual square lattice,  define the height difference,  $h(v) - h(u)$ to be $+1$ (respectively,   $-1$),  if the associated arrow of the primal edge $e$ is crossed from left to right (resp.  from right to left) when going from $u$ to $v$.  Thanks to the ice rule,  the height function on $\Z^2$ is well defined up to an additive constant.  Alternatively,  one may define the height function as an integer valued random surface $h:\Z^2 \to \Z$, such that $|h(u) - h(v) |=1$ for every adjacent $u,v$.  Hence it is also called homomorphism height functions, as it takes even values on even sublattices and odd values on odd sublattices.

One may also define the six-vertex model and its associated height function in a finite domain $D\subset \Z^2$,  as long as the boundary condition is \textit{admissible},  meaning that it extends to at least one six-vertex configuration (and homomorphism height function) in $D$.  

The phase diagram of the six-vertex model (with respect to $a,b,c$) can be found in \cite{Bax}.  It is believed,  that as long as $\Delta := \frac{a^2+b^2-c^2}{2ab} \in (-1,1)$,  the model is disordered,  such that the associated height function delocalizes,  as long as we take the size of the domain to infinity. The conjectured delocalization was only proved in a few limited regions of the phase space. There are extensive studies at the free fermion point $a=b=1, c=\sqrt{2}$,  where a bijection to a dimer model is available,  and the scaling limit of the height function is given by a Gaussian free field \cite{Ken,  Dub}.  The Gaussian free field scaling limit also extends to a neighborhood of the free fermion point,  using renormalization group arguments \cite{GMT}.  Another special case of interest is $a=b=1, c=2$,  where delocalization of the height function is obtained via a coupling to the random cluster model \cite{DCST, GP}.

The case $a=b=c=1$,  which corresponds to the uniform square-ice model (and uniform homomorphism height function),  has received great interest and is the subject of this paper.  Considerable progress has been made in recent years and delocalization of the corresponding height function has been established in \cite{DCHLR,  CPST, Sc}.  The long standing conjecture,  that the  height function has a Gaussian free field (GFF) scaling limit,  still remains open.  The proof in \cite{DCHLR} is quantitative and obtains a logarithmic fluctuation of the height function,  which is coherent with the GFF scaling limit.  However,  the exact asymptotics of the variance,  and the Gaussianity in the limit,  are still missing.  Recently,  the logarithmic fluctuations have been established for all values $a=b=1,  c\in (1,2)$ \cite{DCKMO},  and rotation invariance of any subsequential limit was obtained for $a=b=1, c\in (\sqrt 3, 2)$ in \cite{DCKKMO}.  The GFF scaling limit were also expected in the whole region $a=b=1,  c\in (1,2)$.  Let us also mention that the uniform square-ice model has an explicit bijection to the uniform proper 3-colorings on the dual graph.  For proper 3-colorings in $\Z^2$,  an algebraic rate of mixing and Bernoullicity of the measure,  which is suggested by the logarithmic fluctuation of the square-ice,  was proved in \cite{RS}.

In this paper we give another input to the conjectured Gaussian free field scaling limit for the uniform square ice height function,  by proving a central limit theorem.  %We study the scaling limit of the homomorphism height function that corresponds to the height function of the uniform square ice model.  

To state the result,  let a path ($\times$-path) be a sequence of vertices $v_0, \cdots v_n \in \Z^2$, such that for every $0\le i<n$,  $v_i$ and $v_{i+1}$ are at a Euclidean distance $1$ (resp. $\sqrt{2}$) of each other.  A finite subset $D \subset \Z^2$ whose boundary is a $\times$-circuit of even (resp. odd) vertices is called an even (resp. odd) domain.  Let $\phi_N$ be the uniform distribution on homomorphism height functions defined on an even domain $D_N$, where $ [-N,N]\cap \Z^2 \subset D_N \subset [-N-1,N+1]\cap \Z^2$,  with zero boundary condition,  we prove that 

\begin{theorem}\label{t.clt}
There exist constants $C_1, C_2<\infty$ and $\sigma_N$,  such that $C_1 (
\log N)^\frac 12 \leq \sigma_N \leq C_2 (
\log N)^\frac 12 $,  and such that $\phi_N(0) / \sigma _N $ converges in distribution to a normal random variable. 
\end{theorem}
\begin{remark}
Theorem \ref{t.clt} implies the Gaussianity of the scaling limit.  It is, however, insufficient to give the asymoptotic convergence of the variance $\frac{\sigma_N}{\sqrt{ \log N}}\to g$ for some $g>0$,  which we believe require new ideas.  It is conceivable,  that if such asymptotic convergence were established,  then combining with the Gaussianity proof in this paper one can resolve the conjectured GFF scaling limit.  See Section \ref{s.open} for more discussions.  
\end{remark}
We also obtain a multivariate CLT,  jointly for the height function $(\phi_N(x_{i,N}))_{i=1}^k$ in the bulk of $D_N$. 

\begin{theorem}\label{t.multi}
Let $x_1, \cdots x_k \in (-1,1)\times(-1,1)$,  and suppose $x_{1,N}, \cdots, x_{k,N}$ are points in $D_N$ such that $\frac 1N x_{i,N} \to x_i$ as $N\to \infty$.  There exist constants $C_1, C_2<\infty$ and $\sigma_{1,N}, \cdots , \sigma_{k,N},$,  satisfying $C_1 (
\log N)^\frac 12 \leq \sigma_{i,N} \leq C_2 (
\log N)^\frac 12 $,  for all $i\in {1, \cdots, k}$,  and such that $\phi_N(x_{1,N}) / \sigma _{1,N}, \cdots,   \phi_N(x_{k,N}) / \sigma _{k,N}$ converges in distribution to independent normal random variables. 
\end{theorem}

The proof of Theorem \ref{t.clt} and \ref{t.multi} is based on an application of the martingale central limit theorem. The key inputs include a Russo-Seymour-Welsh estimate for the level line of the height function (established in \cite{DCHLR}),  the FKG inequality for the absolute value height function and the monotone coupling (used in  \cite{DCHLR} and \cite{RS}),  and random walk estimates.  We will summarize these tools in the next section,  and the proof of Theorem \ref{t.clt} and \ref{t.multi} will be given in Section \ref{s.clt} and \ref{s.multi}.  

\section{Preliminaries}

We begin with introducing some notations, following closely with that of \cite{DCHLR, RS}.  Given an even domain $D\subset \Z^2$,  we denote by $\P_D^0$ (resp.  $\E_D^0$) the uniform distribution (resp.  the expectation) on homomorphism height functions in $D$ with zero boundary condition.  Recall that $D_N$ is an even domain such that $ [-N,N]^2\cap \Z^2 \subset D_N \subset [-N-1,N+1]^2\cap \Z^2$.  The following variance bound was established in \cite{DCHLR}. 

\begin{theorem}[Theorem 1.1 of \cite{DCHLR}] \label{t.var}
There exist $c,C\in (0,\infty)$ such that for every $N\ge 1$,  $c\log N\le \E_{D_N}^0[\phi_N(0)^2] \le C\log N$.
\end{theorem}

Theorem 1.1 of \cite{DCHLR} was originally stated for height differences on a torus,  but the same proof also gives the above result on even domains. 

For every $m,n\ge 1$,  we write $\Lambda_{m,n}:=( [-n,n]\times [-m,m])\cap \Z^2$. Let $\H_{h\ge k } (\Lambda_{\rho n, n})$ (resp.  $\H_{h= k } ^\times(\Lambda_{\rho n, n})$) be the event that there is a path (resp.  $\times$-path) with height at least $k$ (resp. equal to $k$) crossing  from the left to the right of $\Lambda_{\rho n, n}$ and contained in $\Lambda_{\rho n, n}$.  The following Russo-Seymour-Welsh theorem was proved in \cite{DCHLR}.

\begin{theorem}[Theorem 1.2 of \cite{DCHLR}]
For every $\eps, R, \rho, k>0$, there exists $c= c(\eps, R, \rho, k)>0$,  such that for every $n$ and every even domain $D \subset D_{Rn}$ such that the distance between $\Lambda_{\rho n, n}$ and $\partial D$ is at least $\eps n$,  
\begin{align*}
c &\leq \P_D^0(\H_{h\ge k } (\Lambda_{\rho n, n})) \leq 1-c, \\
c &\leq \P_D^0(\H_{h= k } ^\times(\Lambda_{\rho n, n})) \leq 1-c.
\end{align*}
\end{theorem}

An immediate consequence to the Russo-Seymour-Welsh theorem is the positive probability of level loop in each annulus with aspect ratio $2$.  For $m<n$ we denote by $A_{m,n}:= D_n \setminus D_m$.  Let $D$ be an even domain containing $D_{2n}$,  and $\mathcal G_n$ be the event that for the square ice height function in $D$ with zero boundary condition,   there is a $|h|=2$ level $\times$-loop contained in $A_{n,2n}$.  

\begin{corollary}\label{c.RSW}
For every $\eps>0$,  there exists $c>0$,  such that for every even domain $D$ containing $D_{2n}$,
\begin{align*}
c &\leq \P_{D}^0(\mathcal G_n) \leq 1-c.
\end{align*}
\end{corollary}
In fact,  we will apply a generalization of this corollary,  stated as Lemma \ref{l.loop} below.  

We also need positive association properties of the square ice.  For a given boundary set $B\subset D$ with $D$ an even domain and $\kappa: B\to \Z$,  define $\text{Hom}(D,B,\kappa)$ to be the set of homomorphisms $h :D\to \Z$ such that $h_v = \kappa_v$ for every $v\in B$.  We call $(B, \kappa)$ a boundary condition and say that the boundary condition is admissible if $\text{Hom}(D,B,\kappa) \neq \emptyset$  is finite.  We write $\P_D^{B, \kappa}$ for the uniform measure on $\text{Hom}(D,B,\kappa)$.

\begin{proposition}[monotonicity for $\phi$ \cite{DCHLR}]
Consider  an even domain $D \subset \Z^2$ and two admissible boundary conditions $(B, \kappa)$ and $(B, \kappa')$ 
satisfying that for every $v\in B$,  $\kappa_v\in [a_v, b_v]$ and $\kappa'_v\in [a'_v, b'_v]$ with $a_v\leq a'_v$ and $b_v\le b'_v$, then 
\begin{itemize}
\item For every increasing function $F$,  $\P_D^{B, \kappa'}[F(\phi)]\ge \P_D^{B, \kappa}[F(\phi)]$
\item For any two increasing functions  $F,G$,  $\P_D^{B, \kappa}[F(\phi)G(\phi)]\ge  \P_D^{B, \kappa}[F(\phi)]\P_D^{B, \kappa}[G(\phi)]$
\end{itemize}
\end{proposition}

We say that the boundary condition $(B, \kappa)$ is $|h|$-adapted,  if there exists a partition $B_{\text{pos}} (\kappa) \cup B_{\text{sym}} (\kappa) $ of $B$,  such that 
\begin{itemize}
\item  for every $v \in B_{\text{pos}} (\kappa)$,  $\kappa_v \in \{0, 1, 2, \cdots\}$;
\item for every $v \in B_{\text{sym}} (\kappa)$,  $\kappa_v = - \kappa_v$.
\end{itemize}

\begin{proposition}[monotonicity for $|\phi|$ \cite{DCHLR}]\label{p.absfkg}
Consider $D \subset \Z^2$ and two admissible boundary conditions $(B, \kappa)$ and $(B, \kappa')$ 
satisfying $B_{\text{pos}} (\kappa) \subset B_{\text{pos}} (\kappa') $  and for every $v\in B$,  $[a_v, b_v]:= \kappa_v \cap \Z_+$ and $[a'_v, b'_v]:= \kappa'_v \cap \Z_+$ satisfy $a_v\leq a'_v$ and $b_v\le b'_v$, then 
\begin{itemize}
\item For every increasing function $F$,  $\P_D^{B, \kappa'}[F(|\phi|)]\ge \P_D^{B, \kappa}[F(|\phi|)]$
\item For any two increasing functions  $F,G$,  $\P_D^{B, \kappa}[F(|\phi|)G(|\phi|)]\ge  \P_D^{B, \kappa}[F(|\phi|)]\P_D^{B, \kappa}[G(|\phi|)]$
\end{itemize}
\end{proposition}

We will only apply the above two propositions with $D$ being an even domain such that for some $n\in\N$,  $ [-n,n]^2\cap \Z^2 \subset D \subset [-n-1,n+1]^2\cap \Z^2$,  and $B = \partial D$  (namely,  $B \subset D$ is the set of vertices that has at least one neighbor in $\Z^2\setminus D$ ).  

Applying Corollary \ref{c.RSW} to the uniform square ice in $D_N$ with zero boundary condition,  we may construct a sequence of  level loops $\L_0, \L_1, \cdots$, of the height function, that are nested $\times$-loops,   such that $\L_0 = \partial D_N$,  and 
$$
\mathcal L_k = \text { outermost loop contained in } \L_{k-1} \text{ such that } \left|\phi_N|_{\mathcal L_k } - \phi_N|_{\mathcal L_{k-1} } \right| =2
$$

We recall the following estimate for the number of loops in an annuli,  stated in \cite{RS} (see also \cite{DCHLR}).

\begin{proposition}\label{l.loop}
There exist constants $C,c > 0 $ such that the following holds.  Let $k \ge 200$, let $a \ge 1$ and let $D$ be a domain.  Let $N$ be the number of level loops contained in 
$A_{k, 2^ak} \subset D$.  Then  $\P_D^0[N< ca] \le Ce^{-ca}$.  

Moreover,   the number of  level loops $N'$ intersecting $A_{k, 2^ak}$ satisfies 
$$
\P_D^0[N'> n]  \leq e^{-c n\log ( \frac na \wedge k)}
$$
for all $n >Ca$.
\end{proposition}

\bigskip

\section{Proof of Theorem \ref{t.clt}} \label{s.clt}
We prove Theorem \ref{t.clt} by constructing a martingale decomposition and applying a martingale difference CLT.  We start by defining the geometric scales and the relevant martingale differences.  For $k = 1, \cdots \log N$,  let $r_k = 2^{-k} N$,  and denote by $B_{k}(x)$ an even domain centered at $x$ and is nested between two boxes $[x-r_k , x+r_k]^2 \cap \Z^2 \subset B_k(x) \subset [x-r_k-1 , x+r_k+1]^2 \cap \Z^2$.  When $x=0$ we simply denote $B_{k}(0)$ by $B_k$,  Also denote by $A_{r_k,r_l}(x): = B_l(x) \setminus B_k(x)$, and  $A_{r_k,r_l}:= A_{r_k,r_l}(0)$. We define a sequence of level loops $\L_0, \L_1, \cdots$, of the height function  such that $\L_0 = \partial D_N$,  and 
$$
\mathcal L_k = \text { outermost loop contained in } \L_{k-1} \text{ such that } \left|\phi_N|_{\mathcal L_k } - \phi_N|_{\mathcal L_{k-1} } \right| =2.
$$

 We also define a sequence of level circuit $\{ \mathcal C_k\}_{k=1}^{\log N}$ of the height function,  such that 
$$
\mathcal C_k = \text { outermost level loop } (\L_j)_{j\in\N} \text{ contained in } B_{k}. 
$$
We may view $\mathcal C_k$ as evaluating $(\L_k)_{k\in \N}$ at random times.  Notice that $\mathcal C_k$ can be defined via a standard exploration process of the height function from $\partial B_k$.  

Given a domain $D\subset \Z^2$,  let $\text{int} (D) := D\setminus \partial D$,  and we define $\mathcal F_k = \sigma \{ \phi(x),  x\notin \text{int }{ \mathcal C_k}  \}$.   Decompose $\phi_N(0)$ into a telescoping sum:
\begin{equation*}
\phi_N(0) = \sum_{k=1}^{\log N} \E[\phi_N(0)|\mathcal F_k ] - \E[\phi_N(0)|\mathcal F_{k-1} ]
\end{equation*}
and denote by the martingale difference $\Delta_k :=  \E[\phi_N(0)|\mathcal F_k ] - \E[\phi_N(0)|\mathcal F_{k-1} ]$.  We notice that $\Delta_k \in 2\Z$. 

We recall the following martingale difference CLT stated in \cite{He}

\begin{theorem} \label{t.martclt}
Let $X_k$ be a martingale difference array satisfying 
\begin{itemize}
%\item $\max_{k \leq n} |X_k|$ is uniformly bounded in $L^2$
%
\item $\E\left[\max_{k \leq n} |X_k| \right] \to 0$
\item $ \sum_{k=1}^n  |X_k|^2 \to 1$ in probability,
\end{itemize}

then $S_n := \sum_{k=1}^n  X_k$ converges in distribution to a standard normal.
\end{theorem}

We  will apply the martingale difference CLT with $X_k := \frac {\Delta_k}{\left(\sum_{k=1}^{\log N}\E[ \Delta_k^2]\right)^\frac12} $. This implies Theorem \ref{t.clt} with $\sigma_N:=\left( \sum_{k=1}^{\log N}\E[ \Delta_k^2]\right)^\frac12$.   
%At the end of this section we will also explain why the proof also applies to $X_k := \frac {\Delta_k}{\var^\frac 12{\phi_N(0)}}= \frac {\Delta_k}{\E[(\sum \Delta_k)^2]^\frac12} $.  

The first assumption in Theorem \ref{t.martclt} follows from the tail estimate for $\Delta_k$, which is a consequence of Proposition \ref{l.loop}.  Apply Proposition \ref{l.loop},  we conclude that for all $k < \log N -200$,  we have $\P_{D_N}^0[\Delta_k> n]\le e^{-c n\log ( n \wedge 2^k)}$.  A union bound (for $\log N$ random variables with exponential tails) thus implies there is some $C<\infty$, such that 
\begin{equation} \label{e.cond1}
\E[\max_{k \leq \log N} |\Delta_k| ]\leq  \E[\max_{k \leq \log N-200} |\Delta_k| ]+ \E[\max_{\log N-200 \leq k \leq \log N} |\Delta_k| ]\leq C \log \log N.
\end{equation}
Another application of Proposition \ref{l.loop} yields $c_1\le \E[ \Delta_k^2] \le c_2$ for some $c_1, c_2\in (0,\infty)$.  Therefore $c_1\log N \le \sum_{k=1}^{\log N}\E[ \Delta_k^2] \le c_2 \log N$,  which implies $\E\left[\max_{k \leq n} |X_k| \right] \to 0$.

  It suffices to check the second assumption in Theorem \ref{t.martclt} .  Since $c_1\log N \le \sum_{k=1}^{\log N}\E[ \Delta_k^2] \le c_2 \log N$, it suffices to show 
\begin{equation}\label{e.cond2}
\frac{1}{\log N}  \sum_{k=1}^{\log N} [\Delta_k^2 -\E[\Delta_k^2]] \to 0 
\end{equation}
in probability. 

%To apply the proposition,  denote by $m_k := \max\{p:  \C_p = \C_k\}$.  Since $m_k -k >p$ implies there is no level circuit in $A_{r_{k+p}, r_k}$,  we apply the proposition to conclude
%\begin{equation*}
%\P^0[m_k -k >p] \le Ce^{-ca}
%\end{equation*}

Let $N(A_{r_{k+ \log\log N}, r_k})$ be the number of level loops contained in $A_{r_{k+ \log\log N}, r_k}$.  Consider the truncated random variable $\tilde\Delta_k := \Delta_k 1_{N(A_{r_{k+ \log\log N}, r_k})\ge 1}$,  which is measurable with respect to the $\sigma$-algebra $\F_{k+\log\log N}$,  and notice that whenever $l-k > \log\log N$,  $\tilde\Delta_k$ only depends on the information in $\sigma \{ \phi(x),  x\notin \text{int }{ \mathcal C_l}  \}$.  We claim it suffices to show 

\begin{equation}\label{e.tilde}
\frac{1}{\log N}  \sum_{k=1}^{\log N} [\tilde \Delta_k^2 - \E[\tilde \Delta_k^2]] \to 0 
\end{equation}
in probability.   Indeed,   it  follows from Proposition \ref{l.loop} that for there exist $c>0$ and $C<\infty$,  such that for all $k \leq \log N-\log\log N-200$,  $\P_{D_N}^0[N(A_{r_{k+ \log\log N}, r_k}) =0 ]
\le  C( \log N)^{-c} $.  Thus
\begin{multline*}
\frac{1}{\log N} \sum_{k=1}^{\log N} \E_{D_N}^0[ |\Delta_k^2 - \tilde \Delta_k^2|]
\leq 
\frac{1}{\log N} \sum_{k=1}^{\log N} \E_{D_N}^0[ \Delta_k^4]^\frac 12 \P_{D_N}^0[N(A_{r_{k+ \log\log N}, r_k}) =0 ]^\frac 12
%\leq 
%\frac{1}{\log N}\left( \sum_{k=1}^{\log N-\log\log N-200}  C( \log N)^{-c} +  \sum_{k=\log N-\log\log N-200}^{\log N}C \right)
\leq
C (\log N)^{-c/2}.
\end{multline*}

The weak law of large numbers \eqref{e.tilde} follows from the following decoupling estimate.

\begin{lemma}\label{l.decouple}
There exist $C, K_0<\infty$,  and $\eps \in (0, \frac 14)$,  such that for every $k -l > \log \log N +1$,  and $\log N -k >K_0$,  we have 
\begin{equation*}
\cov [\tilde \Delta_k^2,  \tilde \Delta_l^2] \leq 
\frac C {(k-l - \log\log N)^{\frac 12 -\eps}}
\end{equation*}
\end{lemma}

Assuming the lemma, a second moment computation gives 

\begin{align*}
&\frac{1}{(\log N)^2} \E \left[  \sum_{k=1}^{\log N} [\tilde \Delta_k^2 - \E[\tilde \Delta_k^2]] \right]^2 \\
&\leq
\frac{2}{(\log N)^2} \sum_{l=1}^{\log N-K_0} \left( C\log\log N +\sum_{k=l+\log\log N +1 }^{\log N-K_0} \frac C {(k-l-\log\log N)^{\frac 12-\eps}}\right)  
+ \frac{CK_0^2}{(\log N)^2} \\
&\leq 
\frac{C}{(\log N)^{{\frac 12-\eps}}} 
\end{align*}
which tends to zero as $N\to \infty$. This yields \eqref{e.tilde},  and thus finishes the proof of Theorem \ref{t.clt}.

Lemma \ref{l.decouple} follows from a monotone coupling of the uniform square ice height functions introduced in \cite{RS}, and an application of the ballot theorem for random walks, stated below. 

\begin{proposition} \label{t.ballot}
Let $X$ be a mean zero random variable with finite second moment,  and $S$ is a simple random walk starting at $0$ and with step size $X$.  There exists $C>0$ such that for all $n\in\N$, 
\begin{equation*}
C^{-1}n^{-\frac12}\leq \P[S_i>0, \forall 0<i<n] \leq Cn^{-\frac12}
\end{equation*}
\end{proposition}

\medskip

\begin{proof}[Proof of the Lemma \ref{l.decouple}]
Assume $k -l > \log \log N +1$ and $\log N -k >K_0$.  We denote by $l^*:= l+ \log\log N$ and $r_{l^*}:=2^{-\log\log N}r_l$,  and $B_{l^*}= B_{r_{l^*}}(0)$. Let $\C_{l+}$ be the innermost level loop outside $r_{l^*}$. We notice that,  on the event $N(A_{r_{l+ \log\log N}, r_l})\ge 1$,  $\C_{l+} \subset A_{r_{l^*}, r_l}$ .  We may write $\E_{D_N}^0\left[\tilde \Delta_k^2  \tilde \Delta_l^2  \right]
= \E_{D_N}^0\left[\E\left[\tilde \Delta_k^2| \C_{l+}\right]  \tilde \Delta_l^2  \right]$.  We claim that there exists $\eps\in (0,\frac 14)$, such that

\begin{equation}\label{e.1}
\E\left[\tilde \Delta_k^2| \C_{l+}\right] = \E_{B_{l^*}}^{0} [\tilde \Delta_k^2] + O\left(\frac 1 {(k-l- \log\log N)^{1/2-\eps}}\right)
\end{equation}
and 
\begin{equation}\label{e.2}
  \E_{B_{l^*}}^{0} [\tilde \Delta_k^2] =  \E_{D_N}^{0} [\tilde \Delta_k^2]+ O\left(\frac 1 {(k-l- \log\log N)^{1/2-\eps}}\right)
\end{equation}
from which the lemma follows. 

The estimates \eqref{e.1} and \eqref{e.2} can be proved by the same argument, so we focus on the first one.  Given the contour $\C_{l+} \subset A_{r_{l^*}, r_l}$,    we construct a coupling of $\P_{B_{l^*}}^{ 0}$ and $\P_{\C_{l+}}^{0}$  in the bulk with a power law rate of mixing (in terms of the geometric scale).  Denote the corresponding height functions by $h_{l^*}^0$  and $h_{\C_{l+}}^0$,  respectively.  By the absolute value FKG stated in Proposition \ref{p.absfkg},   $|h_{\C_{l+}}^0| $ stochastically dominates $ |h_{l^*}^0|$ inside $B_{l^*}$.  This gives rise to a natural monotone coupling of $h_{\C_{l+}}$ and $h_{l^*}$ inside $B_{l^*}$ \cite{RS},  following a standard exploration process of disagreement percolation (see, e.g.,  \cite{BM}),  which we now recall. 
\begin{itemize}
\item We explore the height function $h_{\C_{l+}}^0$ up to the boundary of $B_{l^*}$.  By the domain Markov property, the conditional law of $h_{\C_{l+}}^0$ in $B_{l^*}$ is that of a uniform homomorphism height function in $B_{l^*}$ with this boundary condition.

\item The FKG for absolute value height function implies that $|h_{\C_{l+}}^0|$ restricted in  $B_{l^*}$ stochastically dominates $ |h_{l^*}^0|$.  Consequently, we can construct a coupling between them inside $B_{l^*}$ by simultaneously exploring and revealing the values of both $|h_{\C_{l+}}^0|$ and $ |h_{l^*}^0|$ , vertex by vertex, starting from the boundary of  $B_{l^*}$ inwards, ensuring along the way that $|h_{\C_{l+}}^0|\ge |h_{l^*}^0| $.  We stop the exploration when we discover the outermost height zero loop of $h_{\C_{l+}}^0$ sourrounding $0$ and inside $B_{l^*}$,  and denote this loop by $\L^*$.

\item  The stochastic domination  and the domain Markov property implies that we can jointly sample $|h_{\C_{l+}}^0|$ and $ |h_{l^*}^0| $ inside the domain enclosed by $\L^*$,  according to a single sample from their common distribution. In this case,  we stop the exploration process and obtain a coupling of the two height functions that are equal inside $\L^*$.

\end{itemize}
Suppose the level loop $\L^*$ of $h_{\C_{l+}}^0$ with height 0,  constructed in the exploration process given above,  satisfies $\L^* \subset A_{r_k, r_{l^*}}$,  then the coupling above implies $h_{\C_{l+}}^0=h_{l^*}^0=0$ on $\L^*$, and we may couple the two fields so that they agree inside $\L^*$.  This gives
\begin{equation}
\E\left[\tilde \Delta_k^2 1_{\L^* \subset A_{r_k, r_{l^*}}}| \C_{l+}\right] = \E_{B_{l^*}}^{0} [\tilde \Delta_k^2 1_{\L^* \subset A_{r_k, r_{l^*}}}] 
\end{equation}

To estimate the probability that the coupling fails,  denote by $M$ the total number of level loops contained in $\C_{l+} \setminus B_k$.  We write $M =N+N'$,  where $N$ is the total number of level loops contained in $A_{r_k, r_{l^*},}$,  and $N'$ is the number of level loops that intersects $\partial B_{l ^*}$.  Note that conditioned on $N'$ and $N$,   the height function on the level loops forms a simple random walk in $2\Z$,  and therefore we may use a last visit decomposition (conditioning on the last visit to zero in the first $N'$ steps) to obtain

\begin{equation}
\P[S_{N'+1}S_{N'+2} \cdots S_M \neq 0]
\leq 
\sum_{i=1}^{N'} \P[S_{N'+1}S_{N'+2} \cdots S_M \neq 0| S_i =0]
\leq
 \sum_{i=1}^{N'}  \frac C {(M-i)^{1/2}}
 \leq 
  \frac {CN'} {N^{1/2}},
  \end{equation}
where the second inequality follows from the ballot theorem (Proposition \ref{t.ballot}).   Also notice,   by Lemma \ref{l.loop},  
there exist $c>0$ and $C<\infty$ such that 
\begin{equation}
\P[N' >n] \leq e^{-c n\log ( n \wedge 2^{K_0})}
\end{equation}
and 
\begin{equation}
\P[N < c(k-l-\log\log N)] \leq C e^{- c(k-l-\log\log N)}
\end{equation}
Combining these,  we may choose $K_0 = K_0(c,C)$ large enough,  and a small constant $\eps>0$,  such that the probability of coupling failure is bounded above by 
\begin{multline}
\P[\L^* \not\subset A_{r_k, r_{l^*}} ] \leq  \frac {C'\log(k-l-\log\log N)} {(k-l-\log\log N)^{1/2}} + \P[N' >\log(k-l-\log\log N)]
+ \P[N < c(k-l-\log\log N)] \\
\leq 
\frac {C''} {(k-l-\log\log N)^{1/2-\eps}} .
\end{multline}
Since $\tilde \Delta_k$ have all moments finite,  which follows from Lemma \ref{l.loop},  we apply the H\"{o}lder inequality to obtain \eqref{e.1} (with a larger $\eps$).  
\end{proof}

%\ww{We have thus concluded Theorem \ref{t.clt} with $\sigma_N:=\left( \sum_{k=1}^{\log N}\E[ \Delta_k^2]\right)^\frac12$.  A similar argument would allows us to prove the theorem with $\sigma_N= \var^{\frac 12} \phi_N(0)=\E[(\sum_{k=1}^{\log N} \Delta_k)^2]^\frac12$, as we now sketch.  
%
%Using that $c_1 \log N \leq \var \phi_N(0) \leq c_2 \log N$,  it suffices to prove 
%
%\begin{equation*}
%\frac{1}{\log N}  \sum_{k=1}^{\log N} [\Delta_k^2 - \sum_{l=1}^{\log N} \E[\Delta_k \Delta_l]] \to 0 
%\end{equation*}
%in probability.   A similar calculation that leads to \eqref{e.tilde} also implies that it suffices to study the truncated random variables:
%\begin{equation*}
%\frac{1}{\log N}  \sum_{k=1}^{\log N} [\tilde \Delta_k^2 - \sum_{l=1}^{\log N}  \E[\tilde \Delta_k \tilde \Delta_l]] \to 0 
%\end{equation*}
%thus complete the proof of the theorem. }

\bigskip
\section{The Multivariate CLT} \label{s.multi}
In this section we prove Theorem \ref{t.multi},  by modifying the argument in the previous section.  By definition,  it suffices to show for any $a_1, \cdots a_k \in \R$,  $\sum_{i=1}^k a_i \phi_N(x_{i,N}) / \sigma _{i,N}$ converges in distribution to $\mathcal N (0, \sum_{i=1}^k a_i^2 )$.

Let us set up some notations,  similar to Section \ref{s.clt}.  For each $x_{i,N}\in D_N$,  we define a sequence of level loops $\L_0(x_{i,N}), \L_1,(x_{i,N}) \cdots$, of the height function  such that $\L_0(x_{i,N}) = \partial D_N$,  and 
$$
\mathcal L_m (x_{i,N})= \text { outermost loop contained in } \L_{m-1}(x_{i,N}) \text{ such that } \left|\phi_N|_{\mathcal L_m (x_{i,N})} - \phi_N|_{\mathcal L_{m-1} (x_{i,N})} \right| =2
$$

 We also define,  for each $1\le i \le k$ and $N\in \N$, a sequence of level circuit $\{ \mathcal C_m(x_{i,N})\}_{m=1}^{\log N}$ of the height function,  such that 
$$
\mathcal C_m(x_{i,N}) = \text { outermost level loop } (\L_j)_{j\in\N} \text{ contained in } B_{m} (x_{i,N}).
$$
And we set $\mathcal F_m (x_{i,N})= \sigma \{ \phi_N(x),  x\notin \text{int }{ \mathcal C_m(x_{i,N})}  \}$.   Decompose $\phi_N(x_{i,N})$ into a telescoping sum:
\begin{equation*}
\phi_N(x_{i,N}) = \sum_{m=1}^{\log N} \E[\phi_N(x_{i,N})|\mathcal F_m (x_{i,N})] - \E[\phi_N(x_{i,N})|\mathcal F_{m-1} (x_{i,N})]
\end{equation*}
Introduce the martingale difference $\Delta_{i,m} := \E[\phi_N(x_{i,N})|\mathcal F_m (x_{i,N})] - \E[\phi_N(x_{i,N})|\mathcal F_{m-1} (x_{i,N})]$,  and we set $\sigma_{i,N} = \E_{D_N}^0[(\sum_{m=1}^{\log N} \Delta_{i,m})^2]^\frac12$.  %We will apply Theorem \ref{t.martclt} with $X_{i,m} = \frac {\Delta_{i,m} }{\sigma_{i,N} }$.  

Notice that,  we may write 
\begin{equation*}
\sum_{i=1}^k a_i \phi_N(x_{i,N}) / \sigma _{i,N} =  \sum_{m=1}^{\log N}\sum_{i=1}^k a_i  X_{i,m},
\end{equation*}
where $X_{i,m}:= \frac{\Delta_{i,m}}{ \E_{D_N}^0[(\sum \Delta_{i,m})^2]^\frac12}$.
We apply Theorem \ref{t.martclt}  to the martingale differences $\sum_{i=1}^k a_i  X_{i,m}$.  

The first moment condition in Theorem \ref{t.martclt} follows from Proposition \ref{l.loop} as in \eqref{e.cond1},  and the triangle inequality which gives 
\begin{equation*}
\E_{D_N}^0\left[\max_{m \leq \log N} |\sum_{i=1}^k a_i  X_{i,m}| \right] \leq  \sum_{i=1}^k |a_i |  \E_{D_N}^0\left[\max_{m \leq \log N} |X_{i,m} | \right] \leq \frac{Ck \log\log N}{\log N}
\end{equation*}
which tends to $0$ as $N\to\infty$. 

To check the second moment condition of Theorem \ref{t.martclt} ,  note that 
\begin{equation*}
 \sum_{m=1}^{\log N} \left( \sum_{i=1}^k a_i  X_{i,m} \right) ^2
 = \sum_{m=1}^{\log N}  \sum_{i=1}^k a_i^2  \frac {\Delta_{i,m}^2 }{\E_{D_N}^0[(\sum \Delta_{i,m})^2] }
 + 
2  \sum_{m=1}^{\log N} \sum_{1\le i <j \le k} \frac {a_i a_j\Delta_{i,m} \Delta_{j,m}}{\E_{D_N}^0[(\sum \Delta_{i,m})^2]^\frac 12 \E_{D_N}^0[(\sum \Delta_{j,m})^2]^\frac 12}
\end{equation*}

It follows from the proof of the one point CLT in Theorem \ref{t.clt} (in particular, the proof of \eqref{e.cond2}) that the diagonal terms 
\begin{equation*}
 \sum_{m=1}^{\log N} a_i^2  \frac {\Delta_{i,m}^2 }{\E_{D_N}^0[(\sum \Delta_{i,m})^2] }
\to 
a_i^2
\end{equation*}
in probability.  We claim that the off-diagonal term will not contribute in the limit: for every $1\le i<j \le k$,
\begin{equation}\label{e.offdiag}
\sum_{m=1}^{\log N} \frac {a_i a_j\Delta_{i,m} \Delta_{j,m}}{\E_{D_N}^0[(\sum \Delta_{i,m})^2]^\frac 12 \E_{D_N}^0[(\sum \Delta_{j,m})^2]^\frac 12}
\to 
0
\end{equation}
in probability.   Combine these estimates we conclude Theorem \ref{t.multi}.  

Similar to the previous section, we define the truncated random variable $\tilde\Delta_{i,m}: = \Delta_{i,m} 1_{N(A_{r_{m+ \log\log N}, r_m}(x_{i,N}))\ge 1}$.  It follows from Lemma \ref{l.loop} that all moments of $\tilde\Delta_{i,m}$ and  $\Delta_{i,m} $ are close.  Therefore using $c_1\log N \le \E_{D_N}^0[(\sum \Delta_{j,m})^2] \le C_2 \log N$ it suffices to prove \eqref{e.offdiag} for the random variables $\tilde \Delta_{i,m} $:
\begin{equation}\label{e.offdiagtilde}
\frac 1{\log N} \sum_{m=1}^{\log N} a_i a_j\tilde\Delta_{i,m} \tilde\Delta_{j,m}
\to 
0
\end{equation}
in probability.  

The estimate \eqref{e.offdiagtilde} follows from the decoupling estimate below.  Define 
\begin{equation*}
m_0 := \min\{ m\in \N: \text{ such that } \forall 1\le i <j\le k,  B_m(x_{i,N})\cap B_m(x_{j,N}) =\emptyset  \}
\end{equation*}
being the first scale that separates the points $(x_{i,N})_{i=1}^k$.  

\begin{lemma}\label{l.cross}
Fix $1\le i<j \le k$.  There exists $C, K_0<\infty$,  and $\eps \in (0,\frac 14)$,  such that for every $m -m_0> \log \log N +1$,  and $\log N -m >K_0$,  we have 
\begin{equation}\label{e.cross1}
\E_{D_N}^0[\tilde \Delta_{i,m}  \tilde \Delta_{j,m}] \leq 
\frac C {(m-m_0 - \log\log N)^{\frac 12 -\eps}}
\end{equation}
Moreover, for every $m_0 \le l < m \le \log N -K_0$ such that $m -l> \log \log N +1$, we have 
\begin{equation}\label{e.cross2}
\cov[\tilde \Delta_{i,m}  \tilde \Delta_{j,m};  \tilde \Delta_{i,l}  \tilde \Delta_{j,l}] \leq 
\frac C {(m-l - \log\log N)^{\frac 12 -\eps}}
\end{equation}
\end{lemma}

Assuming the lemma,  we may estimate the second moment in \eqref{e.offdiagtilde} as 
\begin{align*}
&\frac 1{(\log N)^2} \sum_{m=1}^{\log N}  a_i^2 a_j^2 \E_{D_N}^0[\tilde \Delta_{i,m}^2  \tilde \Delta_{j,m}^2]
+ 
\frac 1{(\log N)^2} \sum_{l, m=1}^{\log N} a_i^2 a_j^2\cov[\tilde \Delta_{i,m}  \tilde \Delta_{j,m};  \tilde \Delta_{i,l}  \tilde \Delta_{j,l}]  \\
&\leq
C_1 \frac 1{\log N}  + \frac 1{\log N}   \left( C_2 \log\log N + \sum_{l=1}^{\log N-K_0} \sum_{m= l+\log\log N+2}^{\log N-K_0}  \frac C {(m-l - \log\log N)^{\frac 12 -\eps}} \right) \\
&+ \frac{CK_0^2}{(\log N)^2}
\leq 
\frac{C}{(\log N)^{{\frac 12-\eps}}} 
\end{align*}
which tends to 0 as $N \to \infty$.

Finally we explain how Lemma \ref{l.cross} follows from a similar proof of Lemma \ref{l.decouple}.

\begin{proof}[Proof of Lemma \ref{l.cross}]
Denote by $m_0^*= m_0+\log\log N$,  and let $\C_{m_0^*+}(x_{i,N})$ be the innermost level loop outside $B_{r_{m_0^*}}(x_{i,N})$. We notice that,  $\C_{m_0^*+}(x_{i,N})\subset A_{r_{m_0^*}, r_{m_0}}(x_{i,N})$ .  We may write $\E_{D_N}^0[\tilde \Delta_{i,m}  \tilde \Delta_{j,m}]
= \E\left[\E\left[\tilde \Delta_{i,m} | \C_{m_0^*+}(x_{i,N})\right]   \tilde \Delta_{j,m}  \right]$.  Notice that, the same  monotone coupling argument that leads to \eqref{e.1} gives that for every $m -m_0> \log \log N +1$,  and $\log N -m >K_0$,  

\begin{equation}
\E\left[\tilde \Delta_{i,m} |\C_{m_0^*+}(x_{i,N})\right]   = \E_{B_{m_0^*}}^{0} [\tilde \Delta_{i,m}] + O(\frac 1 {(m-m_0- \log\log N)^{1/2-\eps}}) 
= O(\frac 1 {(m-m_0- \log\log N)^{1/2-\eps}}) 
\end{equation}
 This yields \eqref{e.cross1}.  
 
 To prove \eqref{e.cross2},  denote by $l^*= l+\log\log N$,  and let $\C_{l^*+}(x_{i,N})$ be the innermost level loop outside $B_{r_{l^*}}(x_{i,N})$. We notice that,  $\C_{l+} \subset A_{r_{l^*}, r_l}(x_{i,N})$ .  We may write $\E_{D_N}^0[\tilde \Delta_{i,m}  \tilde \Delta_{j,m}  \tilde \Delta_{i,l}  \tilde \Delta_{j,l}] 
= \E\left[\E\left[ \tilde \Delta_{i,m}|\C_{l^*+}(x_{i,N})\right]   \tilde \Delta_{j,m}  \tilde \Delta_{i,l}  \tilde \Delta_{j,l} \right]$. 
 The same argument that leads to \eqref{e.1} gives that for $m -l> \log \log N +1$

\begin{equation}
\E\left[\tilde \Delta_{i,m}|\C_{l^*+}(x_{i,N})\right]   
= O(\frac 1 {(m-l- \log\log N)^{1/2-\eps}}) 
\end{equation}
Combining this with \eqref{e.cross1} we obtain \eqref{e.cross2}.  
\end{proof}

\section{Open questions} \label{s.open}

\textbf{Asymptotic variance. }  As mentioned in the introduction,  it remains open to prove  $\frac{\sigma_N}{\log ^\frac 12 N} \to g$,  for some $g>0$.  From our point of view, this remains the main challenge to obtain the conjectured GFF scaling limit for the uniform square ice height function. 

The proof presented in Section \ref{s.clt} suggests that asymptotic convergence of the variance would follow from 
\begin{equation*}
\frac{1}{\log N}   \sum_{k=1}^{\log N} \E[\tilde \Delta_k^2] \to g
\end{equation*}
Proving this claim requires new ideas,  and essentially requires some approximate scale invariance. 

\medskip

\textbf{Mixing rate for the height function.}  We obtained an upper bound for the mixing rate for the gradients of the height function,  which is algebraic in terms of the geometric scale (see Lemma \ref{l.decouple}).   We expect that the precise mixing rate is algebraic in the physical scale,    similar to the gradients of DGFF.

\medskip
\textbf{Central limit theorem for generalised solid-on-solid models.}   The (generalised) solid-on-solid models is an important class of random surface models which exhibit a roughening transition in two dimensions.  To define it,  let $\phi: \Z^2 \to \Z$ be integer valued height functions sampled from the Gibbs weight 

\begin{equation*}
d\mu(\phi) = \text{normalization} \times \exp\left( -\beta \sum_{x\sim y}V(\phi(x) -\phi(y)) \right)
\end{equation*}
where $V: \Z\to \R$ is a symmetric convex potential.  Prominent examples include $V(x) = \frac{x^2}2$ (Discrete Gaussian) and $V(x) =|x|$.   For $\beta$ sufficiently small,  it is conjectured that the scaling limit of the height function is given by a GFF.  The conjecture is still widely open (except for the quadratic case).   The seminal work of Fr\"{o}hlich and Spencer \cite{FS} established the logarithmic fluctuation of the Discrete Gaussian and the absolute value potential model.  There are extensive recent results on solid-on-solid models: the GFF scaling limit for the Discrete Gaussian case was proved in \cite{BPR1, BPR2},  and the (qualitative) delocalization were established for general potential \cite{Lam, LO, AHPS,EL}.  

It is conceivable that the method of this paper also extends to the Gaussianity of limit for generalised solid-on-solid models,  which remains open except for the quadratic case.   Indeed,  the FKG inequality for the absolute value height function has been established for the solid-on-solid models (see e.g.,  \cite{LO}).  If the Russo-Seymour-Welsh estimate can be proved then one can carry out the arguments of this paper to obtain the analogues of Theorem \ref{t.clt}.

 \subsection*{Acknowledgments}
We thank Jian Ding and Chenlin Gu for helpful discussions,  and Jian Ding for helpful comments on previous versions of the manuscript. The research was partially supported by the National Key R\&D Program of China and a grant from Shanghai Ministry Education Commission. 

{\small
\bibliographystyle{amsplain}
\bibliography{6vertex}
}
\end{document}